\documentclass[12pt,reqno]{amsart}
\usepackage[cp1251]{inputenc}
\usepackage[T2A]{fontenc}
\usepackage[english]{babel}
\usepackage{geometry}
\usepackage{amsmath, amsthm, amscd, amsfonts, amssymb, graphicx, color}
\usepackage[bookmarksnumbered, colorlinks, plainpages]{hyperref}

\numberwithin{equation}{section}

\newtheorem{theorem}{Theorem}[section]
\newtheorem{corollary}[theorem]{Corollary}

\theoremstyle{remark}
\newtheorem{remark}[theorem]{Remark}

\theoremstyle{definition}
\newtheorem{definition}[theorem]{Definition}
\newtheorem*{main-definition}{Main Definition}

\begin{document}

\title[A new look at old theorems of Fej\'{e}r and Hardy]{A new look at old theorems of Fej\'{e}r and Hardy}


\author[V. Mikhailets]{Vladimir Mikhailets}

\address{Institute of Mathematics of the Czech Academy of Sciences, \v{Z}itn\'{a} 25, 11567, Praha, Czech Republic; Institute of Mathematics of the National Academy of Sciences of Ukraine, Tereshchen\-kivs'ka 3, Kyiv, 01024, Ukraine}

\email{mikhailets@imath.kiev.ua; mikhailets@math.cas.cz}


\author[A. Murach]{Aleksandr Murach}

\address{Institute of Mathematics of the National Academy of Sciences of Ukraine, Tereshchen\-kivs'ka 3, Kyiv, 01024, Ukraine}

\email{murach@imath.kiev.ua}


\author[O. Tsyhanok]{Oksana Tsyhanok}

\address{National Technical University of Ukraine "Igor Sikorsky Kyiv Polytechnic Institute", Prospect Peremohy 37, 03056, Kyiv-56, Ukraine}

\email{tsyhanok.ok@gmail.com}


\subjclass[2010]{42A20, 42A24, 40E05}




\keywords{Trigonometric series, Fejer theorem, Hardy theorem, Cesaro summability, homogeneous Banach space}

\begin{abstract}
The article studies the convergence of trigonometric Fourier series via a new Tauberian theorem for Ces\`{a}ro summable series in abstract normed spaces. This theorem generalizes some known results of Hardy and Littlewood for number series. We find sufficient conditions for the convergence of trigonometric Fourier series in homogeneous Banach spaces over the circle. These conditions are expressed in terms of the Fourier coefficients and are weaker than Hardy's condition. We give a description of all Banach function spaces given over the circle and endowed with a norm been  equivalent to a norm in a homogeneous Banach space. We study interpolation properties of such spaces and give new examples of them. We extend the classical Fej\'{e}r theorem on the uniform Ces\`{a}ro summability of the Fourier series on sets by means of a refined version of Cantor's theorem on the uniform continuity of a mapping between metric spaces. We also generalize the classical Hardy theorem on the uniform convergence of the Fourier series on sets.
\end{abstract}

\maketitle

\section{Introduction}\label{sect1}

The theory of trigonometric Fourier series has a long history and occupies a special place in analysis. Problems arising in this theory have been the source for new ideas and methods, which found applications to many other mathematical problems. The main questions of the theory concern the convergence or summability of trigonometric series on certain sets with respect to norms in various function spaces. The spaces $C$ and $L^p$ are the most studied in this regard (see, e.g., the well-known monographs \cite{Bari, Edwards, Zygmund}). The results obtained have become an important benchmark for the theory of general orthogonal series \cite{Alexits61, KashinSaakyan89} and the theory of expansions in eigenfunctions of self-adjoint differential operators \cite{AlimovIlinNikishin76, AlimovIlinNikishin77}. The achievements of the theory of Fourier series have also laid the foundation for creating and developing the wavelet theory and frame theory.

Among the summation methods, the most studied ones are the regular methods of Ces\`{a}ro, de la Vall\'{e}e-Poussin, and Abel, which arose originally  for the summation of divergent number series \cite{Hardy}. Later on, these methods found numerous applications, specifically to the summation of  Fourier series. Apparently, the most famous result here is the Fej\'{e}r fundamental theorem concerning Cesaro's $(C,1)$ method.

In Section~\ref{sect2}, we establish sufficient conditions for the convergence of $(C,1)$ summable series in an arbitrary complex or real linear normed space. The space can be incomplete. Theorem~\ref{thm:theor1} proved there contains some known results from Hardy's book~\cite{Hardy} in the case of number series. The theorem is a base for our further results.

In Section~\ref{sect4}, we obtain a generalization of the classical Hardy theorem on the uniform convergence of the trigonometric series of a continuous periodic function whose Fourier coefficients admit the estimate $\mathrm{O}(n^{-1})$. The generalization means that we replace $C(\mathbb{T})$ with an arbitrary homogeneous Banach space and weaken the condition on the Fourier coefficients. Properties of homogeneous Banach space are previously studied in Section~\ref{sect3}. We also give new examples of such spaces.

Section~\ref{sect5} establishes a refined form of Cantor's theorem on the uniform continuity of a mapping between metric spaces.

The final Section~\ref{sect6} proves extended versions of the known theorems of Fej\'{e}r and Hardy on the uniform summability and convergence of trigonometric Fourier series. We show that these theorems remain valid if we replace the circle $\mathbb{T}$ or its closed subinterval with an arbitrary closed subset of $\mathbb{T}$. Moreover, we weaken Hardy's condition on the Fourier coefficients in the convergence theorem.

\section{A Tauberian theorem}\label{sect2}

Let $X$ be a real or complex linear space endowed with a norm $\|\cdot\|$. We prove a new Tauberian-type theorem concerning series in $X$. It generalizes the classical Hardy and Hardy--Littlewood theorems in the case of number series and the $(C,1)$ summation method.

Given a sequence $(u_k)_{k=0}^{\infty}$ in $X$, we use the standard notation
\begin{gather*}
S_n:=\sum_{k=0}^n u_k,\\
\sigma_n:=\frac{S_0+S_1+ ... +S_n}{n+1},\\
\sigma_{n,m}:=\frac{S_n+S_{n+1}+...+S_{n+m-1}}{m},
\end{gather*}
with the integers $n\geq0$ and $m\geq1$. Recall that the series $\sum_{k=0}^\infty u_{k}$ is called a $(C,1)$ summable in $X$ to a vector $S\in X$ if
\begin{equation}\label{Cesaro-summability}
\sigma_n\to S\;\;\mbox{in}\;\;X\;\;\mbox{as}\;\;n\to\infty.
\end{equation}
The $(C,1)$ summation method is regular in $X$; i.e., if the series satisfies \eqref{Cesaro-summability} and is convergent in $X$, then $S$ is the sum of this series.

\begin{theorem}\label{thm:theor1}
Let $1\leq p<\infty$. Suppose that a series $\sum_{k=0}^\infty u_{k}$ is $(C,1)$ summable in $X$ to a vector $S\in X$ and that
\begin{equation}\label{eq: eq1}
\sum_{k=n+1}^\infty \|u_{k}\|^p = O\left( \frac{1}{n^{p-1}} \right)\;\;\,\mbox{as}\;\;\,n\to\infty.
\end{equation}
Then this series converges to $S$ in $X$.
\end{theorem}

\begin{proof}
If $p=1$, condition \eqref{eq: eq1} means the absolute convergence of the series, which implies its convergence in the completion $\widetilde{X}$ of the normed linear space $X$. Owing to the hypothesis, this series is $(C,1)$ summable in $\widetilde{X}$ to the vector $S\in X\subset{\widetilde{X}}$. Hence, $S$ is the sum of the series due to the regularity of the $(C,1)$ summation method in the Banach space $\widetilde{X}$.

Suppose now that $p>1$. It follows from the known equality
$$
\sigma_{n,m} = S_n + \sum_{j=n+1}^{n+m} \left(1-\frac{j-n}{m} \right)u_j $$
that
$$
\| \sigma_{n,m} - S_n \| \leq \sum_{j=n+1}^{n+m} \left(1-\frac{j-n}{m} \right) \|u_j\|
$$
whenever the integers $n\geq1$ and $m\geq2$. By H\"older's inequality, we get
\begin{align*}
\|\sigma_{n,m}-S_n\| &\leq  \left( 	\sum_{j=n+1}^{n+m-1} \left(1-\frac{j-n}{m} \right)^q \right)^{1/q} \left( \sum_{j=n+1}^{n+m-1} \|u_j\|^p \right)^{1/p}\\
&\leq\left( m-1 \right)^{1/q} \left( \sum_{j=n+1}^{n+m-1} \|u_j\|^p \right)^{1/p},
\end{align*}
with $1/p+1/q=1$. Hence, by hypothesis \eqref{eq: eq1}, there exists a number $M>0$ such that
$$
\| \sigma_{n,m} - S_n \| \leq  \left( m-1 \right)^{1/q} \left( \frac{M}{n^{p-1}} \right)^{1/p}
$$
whenever $n\geq1$ and $m\geq1$. (The $m=1$ case is covered because $\sigma_{n,1}=S_n$.)

Choose a number $\varepsilon>0$ arbitrarily and put
$$
m:= m(n):= [n \varepsilon^q]+1.
$$
(As usual, $[a]$ denotes the integral part of a real number $a$.) Then
$$
\|\sigma_{n,m} - S_n \| \leq  [n \varepsilon^q]^{1/q} \cdot
\frac{M^{1/p}}{n^{(p-1)/p}} \leq \varepsilon M^{1/p}
$$
whenever the integer $n\geq1$. Since $n/m<\varepsilon^{-q}$, we have
\begin{align*}
\sigma_{n,m} &= \left( 1+ \frac{n}{m}\right) \sigma_{n+m-1} - \frac{n}{m}\, \sigma_{n-1}\\
&=\sigma_{n+m-1} + \frac{n}{m} \left( \sigma_{n+m-1} - \sigma_{n-1} \right)  \to S\;\;\,\mbox{in}\;\;\, X\;\;\,\mbox{as}\;\;\,n \to \infty,
\end{align*}
due to the hypothesis. Hence, there exists an integer $n_0=n_0(\varepsilon)\geq1$ such that $$
\| S- \sigma_{n,m} \| < \varepsilon\quad\mbox{whenever}\quad n>n_0.
$$
Then
$$
\| S-S_n \| \leq \|S- \sigma_{n,m} \| + \| \sigma_{n,m} -S_n \| \leq \varepsilon +\varepsilon M^{1/p}
$$
whenever $n>n_0$. This implies the conclusion due to the arbitrariness of positive $\varepsilon$.
\end{proof}

\begin{corollary}\label{n1}
If a series $\sum_{n=0}^\infty u_{n}$ is $(C,1)$ summable in $X$ and satisfies
\begin{equation}\label{Hardy-cond}
\|u_{n}\|= O\left(\frac{1}{n}\right)\;\;\,\mbox{as}\;\;\,n\to\infty,
\end{equation}
then this series is convergent in $X$.
\end{corollary}	

In the case of number series, this result is due to Hardy \cite[Section~6.1, Theorem~63]{Hardy} (see also \cite[Chapter~III, Theorem~1.26]{Zygmund}).

\begin{proof}
If there exists a number $M>0$ such that $\|u_{n}\|\leq M/n$ for each $n\geq1$, then, whatever $p>1$, we have
\begin{equation*}
\sum_{k=n+1}^\infty\|u_{k}\|^p\leq M\sum_{k=n+1}^\infty \frac{1}{k^p} \leq M\int\limits_n^\infty\frac{dx}{x^p}=O\left(\frac{1}{n^{p-1}}\right)
\;\;\,\mbox{as}\;\;\,n\to\infty.
\end{equation*}
Thus, Hardy's condition \eqref{Hardy-cond} implies \eqref{eq: eq1} for any $p>1$.
\end{proof}

\begin{corollary}\label{n2}
Let $1\leq p<\infty$. If a series $\sum_{n=0}^\infty u_{n}$ is $(C,1)$ summable in $X$ and satisfies
\begin{equation}\label{corol-sum-cond}
\sum_{n=0}^\infty \ n^{p-1} \|u_{n}\|^p < \infty,
\end{equation}
then this series is convergent in $X$.
\end{corollary}

As for number series, this is Hardy and Littlewood's result  \cite[Section~6.3, Theorem~69]{Hardy}.

\begin{proof}
Indeed, if \eqref{corol-sum-cond} holds true, then
\begin{equation*}
n^{p-1}\sum_{k=n+1}^\infty\|u_{k}\|^p\leq
\sum_{k=n+1}^\infty k^{p-1}\|u_{k}\|^p\to 0
\;\;\,\mbox{as}\;\;\,n\to\infty,
\end{equation*}
which implies the stronger condition
\begin{equation*}
\sum_{k=n+1}^\infty \|u_{k}\|^p = o\left( \frac{1}{n^{p-1}} \right)\;\;\,\mbox{as}\;\;\,n\to\infty
\end{equation*}
than \eqref{eq: eq1}.
\end{proof}

Corollaries \ref{n1} and \ref{n2} are not comparable. Theorem \ref{thm:theor1} is a stronger result than these corollaries. Note that conditions \eqref{eq: eq1} and \eqref{corol-sum-cond} are equivalent if $p=1$.

\begin{remark}
It follows from from the proof of Theorem~\ref{thm:theor1} that this theorem remains valid in the $p>1$ case if $\|\cdot\|$ is a seminorm.
\end{remark}

\section{Homogeneous function spaces}\label{sect3}

We study the convergence of trigonometric Fourier series of $2\pi$-periodic complex-valued functions in homogeneous Banach spaces. Let $\mathbb{T}$ denote the additive quotient group $\mathbb{R}/2\pi\mathbb{Z}$, with $2\pi\mathbb{Z}$ being a subgroup of all real numbers devisible by $2\pi$. We naturally interpret functions on $\mathbb{T}$ as $2\pi$-periodic functions on the closed interval $[-\pi,\pi]$ or as functions on a circle of radius $1$. As usual, $L^1(\mathbb{T})$ denotes the complex Banach space of all integrable functions on $\mathbb{T}$ with respect to the Lebesgue measure. We consider the following norm in this space:
\begin{equation*}
\|f\|_{L^1(\mathbb{T})}:=\frac{1}{2\pi}\int\limits_{-\pi}^{\pi}|f(s)|ds.
\end{equation*}

Let us recall the following notion \cite[Chapter~I, Subsection~2.10]{Katznelson}:

\begin{definition} A complex Banach space $X$ endowed with a norm $\|\cdot\|$ is called \emph{homogeneous} on $\mathbb{T}$ if $X$ is a linear subspace of $L^1(\mathbb{T})$ and satisfies the following three conditions:
\begin{itemize}
\item[($H_1$)] $\|f\|_{L^1(\mathbb{T})}\leq\|f\|$ whenever $f\in X$;
\item[($H_2$)] $f(\cdot+a)\in X$ and $\|f(\cdot+a)\|=\|f(\cdot)\|$ whenever $f\in X$ and $a\in\mathbb{T}$;
\item[($H_3$)] $\|f(\cdot+b)-f(\cdot+a)\|\to 0$ as $b\to a$ whenever $f\in X$ and $a\in\mathbb{T}$.
    \end{itemize}
\end{definition}

Since we consider homogeneous Banach spaces only on $\mathbb{T}$, we will omit the phrase "on $\mathbb{T}$"\ when referring to such spaces.

The complex Banach space $C(\mathbb{T})$ of all continuous functions on $\mathbb{T}$ is homogeneous. Some important properties of Fourier series in $C(\mathbb{T})$ are preserved for homogeneous Banach spaces and are invariant with respect to a choice of an equivalent norm. The following result gives a description of such norms.

\begin{theorem}\label{thm:theor2_dodana}
Let a Banach space $X$ be a linear subspace of $L^1(\mathbb{T})$ such that $f(\cdot+a)\in X$ whenever $f\in X$ and $a\in\mathbb{T}$. Then the norm $\|\cdot\|_{X}$ in $X$ is equivalent to a certain norm $\|\cdot\|$ subject to conditions $(H_{1})$--$(H_{3})$ if and only if $\|\cdot\|_{X}$ satisfies the following two conditions:
\begin{itemize}
\item [($H_1^*$)] there exists a number $c>0$ such that $\|f\|_{L^1(\mathbb{T})}\leq c\,\|f\|_{X}$ for every $f\in X$;
\item [($H_3^*$)] $\|f(\cdot+a)-f(\cdot)\|_{X}\to0$ as $a\to0$ for every $f\in X$.
\end{itemize}
\end{theorem}

\begin{proof} The necessity is obvious. Let us prove the sufficiency. Suppose that the norm $\|\cdot\|_{X}$ satisfies conditions ($H_1^*$) and ($H_3^*$). Given $a\in\mathbb{T}$, we consider the translation operator $T_{a}:f(\cdot)\mapsto f(\cdot+a)$ with $f\in X$. It follows from condition ($H_1^*$) that this operator is closed in the Banach space $X$ endowed with the norm $\|\cdot\|_{X}$. Indeed, let a sequence $(f_{k})\subset X$ satisfy $\|f_{k}\|_{X}\to 0$ and $\|T_{a}f_{k}-g\|_{X}\to 0$ for certain $g\in X$ as $k\to\infty$. Then, by condition ($H_1^*$), we get $f_{k}\to 0$ and $T_{a}f_{k}\to g$ in $L^1(\mathbb{T})$ as $k\to\infty$, which implies $g=0$. This means that the operator $T_{a}$ is closed in $X$. Hence, being defined on the whole Banach space $X$, this operator is bounded on $X$ due to the closed graph theorem.

Given $f\in X$, we see by condition ($H_3^*$) that
\begin{equation}\label{convergence-T}
\|T_{b}f-T_{a}f\|_{X}=\|T_{b-a}T_{a}f-T_{a}f\|_{X}\to0
\;\;\mbox{as}\;\;b\to a\;\;\mbox{whenever}\;\;a\in\mathbb{T}.
\end{equation}
Hence, the function $\|T_{a}f\|_{X}$ of $a$ is continuous on the compact set $\mathbb{T}$, which implies its boundedness on $\mathbb{T}$.
By the Banach--Steinhaus theorem, there exists a number $m\geq1$ such that $\|T_{a}\|\leq m$ for every $a\in\mathbb{T}$; here, $\|T_{a}\|$ denotes the norm of the bounded operator $T_{a}$ on $X$.

Putting
\begin{equation*}
\|f\|=c\cdot\max_{a\in \mathbb{T}}\|T_a f\|_{X}
\quad\mbox{whenever}\quad f\in X,
\end{equation*}
we see that $\|\cdot\|$ is a norm in $X$ and satisfies condition ($H_2$). Owing to condition ($H_1^*$), this norm satisfies condition ($H_1$). The norms $\|\cdot\|$ and $\|\cdot\|_{X}$ are equivalent on $X$; namely:  $c\,\|f\|_{X}\leq \|f\|\leq c\,m\|f\|_{X}$ whenever $f\in X$. Then the norm $\|\cdot\|$ satisfies condition ($H_3$) by \eqref{convergence-T}.
\end{proof}

The property of Banach spaces to be homogeneous are preserved (up to equivalence of norms) under various methods of interpolation between normed spaces, which yields important examples of homogeneous spaces. We consider the two most commonly used methods. Let $X_0$ and $X_1$ be Banach spaces, and suppose that they are continuously embedded in ${L^1(\mathbb{T})}$. As usual, $(X_0,X_1)_{\theta,q}$ and $[X_0,X_1]_{\theta}$ denote the Banach spaces obtained respectively by the real method and complex method of interpolation between the spaces $X_0$ and $X_1$, with $\theta$ and $q$ being  interpolation parameters (see, e.g., \cite[Sections 1.3.2 and 1.9.2]{Triebel}). Note that the interpolation spaces $(X_0,X_1)_{\theta,q}$ and $[X_0,X_1]_{\theta}$ are linear subspaces of $L^1(\mathbb{T})$.

\begin{theorem}\label{thm:theor3_dodana}
If the Banach spaces $X_0$ and $X_1$ are homogeneous, then the interpolation spaces $(X_0,X_1)_{\theta,q}$ and
$[X_0,X_1]_{\theta}$ are homogeneous up to equivalence of norms whenever $\nobreak{0<\theta<1}$ and $1\leq q<\infty$.
\end{theorem}

\begin{proof}
Suppose that the Banach spaces $X_0$ and $X_1$ are homogeneous. Let $X$ denote either $(X_0,X_1)_{\theta,q}$ or $[X_0,X_1]_{\theta}$ for $\theta$ and $q$ indicated. Let us prove that $X$ satisfies the hypotheses of Theorem~\ref{thm:theor2_dodana}. Whatever $a\in\mathbb{T}$, the translation operator $T_{a}:f(\cdot)\mapsto f(\cdot+a)$, with $f\in L^1(\mathbb{T})$, sets the isometric operators on $X_0$ and $X_1$. Hence, $T_{a}$ is a bounded operator on the interpolation space $X$ and
\begin{equation}\label{estimate-T-operator}
\|T_{a}:X\to X\|\leq \|T_{a}:X_0\to X_0\|^{1-\theta}\,\|T_{a}:X_1\to X_1\|^{\theta}\leq1,
\end{equation}
due to \cite[Theorems 1.3.3 (a) and 1.9.3 (a)]{Triebel}. Specifically,
$f(\cdot+a)\in X$ whenever $f\in X$. Since the identity operator on $L^1(\mathbb{T})$ sets the bounded embedding operators $X_j\hookrightarrow L^1(\mathbb{T})$, with $j=0,1$, then it also sets the bounded embedding operator $X\hookrightarrow L^1(\mathbb{T})$. Thus, the norm in $X$ satisfies condition ($H_1^*$).

As is known \cite[Theorems 1.3.3 (g) and 1.9.3 (f)]{Triebel}, there exists a number $c=c(\theta,q)>0$ such that
\begin{equation*}
\|f\|_{X}\leq c\,\|f\|_{0}^{1-\theta}\|f\|_{1}^{\theta}\quad\mbox{for every}\quad f\in X_0\cap X_1,
\end{equation*}
with $\|\cdot\|_{0}$ and $\|\cdot\|_{1}$ denoting the norms in $X_0$ and $X_1$, resp. Hence, the norm in $X$ satisfies condition ($H_3^*$) provided that $f\in X_0\cap X_1$. Since the set $X_0\cap X_1$ is dense in $X$ \cite[Theorems 1.6.2 and 1.9.3 (c)]{Triebel} and because of \eqref{estimate-T-operator}, this norm satisfies condition ($H_3^*$) whenever $f\in X$.

Now the conclusion of Theorem~\ref{thm:theor3_dodana} follows from Theorem~\ref{thm:theor2_dodana}.
\end{proof}

We mention the following examples of homogeneous Banach spaces up to equivalence of norms:
\begin{itemize}
\item [1.] The \emph{space} $C(\mathbb{T})$ of \emph{continuous} functions.

\item[2.] The \emph{space} $C^{(n)}(\mathbb{T})$ of $n\in \mathbb{N}$ times \emph{continuously differentiable} $2\pi$-periodic functions endowed with the norm
    \begin{equation*}
    \|f\|_{C^{(n)}(\mathbb{T})}:=\max_{0\leq k \leq n}\|f^{(k)}\|_{C(\mathbb{T})}.
    \end{equation*}
    This space is homogeneous and separable.

\item[3.] The \emph{Lebesgue spaces} $L^p(\mathbb{T})$, with
$1\leq p<\infty$.

\item[4.] The \emph{Orlicz spaces} $L^{^*}_M(\mathbb{T})$, with the function $M(\cdot)\geq0$ being continuous, even, and convex on $\mathbb{R}$ and satisfying the conditions
\begin{equation*}
\lim_{x\to0}\frac{M(x)}{x}=0\quad\mbox{and}\quad  \lim_{x\to\infty}\frac{M(x)}{x}=\infty;
\end{equation*}
see, e.g., \cite[\S~9]{KrasnoselskiiRutickii} or \cite[Section~4.6]{PickKufnerJohnFucik}. All separable Orlicz spaces are homogeneous.

\item[5.] \emph{Symmetric spaces.} The norm in any separable symmetric space $X$ over $\mathbb{T}$ is invariant and continuous with respect to translations
    \cite[Chapter~II, \S~4, Subsections 1 and~5]{KreinPetuninSemenov}. Hence, $X$ is homogeneous (up to a proportional norm). Specifically, the \emph{endpoint Lorentz space} $\Lambda_\psi(\mathbb{T})$ is homogeneous provided that the concave increasing function $\psi(t)\geq0$ of $t\geq0$ satisfies $\psi(t)\to0$ as $t\to0$ and $\psi(t)\to\infty$ as $t\to\infty$; see, e.g.,
    \cite[Chapter~II, \S~5, Subsection~5]{KreinPetuninSemenov} and \cite[Section~7.10]{PickKufnerJohnFucik}.

\item[6.] The \emph{Lorentz spaces} $L^{p,q}(\mathbb{T})$, with $1<p<\infty$ and $1\leq q<\infty$ (see, e.g.,
    \cite[Subsection 1.18.6]{Triebel} or
    \cite[Section 8.1]{PickKufnerJohnFucik}).
    They are obtained by the real interpolation between Lebesgue spaces; namely: if $1<p_{0}<\infty$, $1<p_{1}<\infty$, $0<\theta<1$, and $1/p=(1-\theta)/p_{0}+\theta/p_{1}$, then
    \begin{equation*}
    L^{p,q}(\mathbb{T})=(L^{p_0}(\mathbb{T}),L^{p_1}(\mathbb{T}))_{\theta,q}
    \end{equation*}
    up to equivalence of norms \cite[Theorem 1.18.6/2]{Triebel}. These Lorentz spaces are symmetric and separable.

\item[7.] The \emph{Sobolev spaces}   $W_p^n(\mathbb{T})=H_p^n(\mathbb{T})$, with
   $1\leq p<\infty$ and $n\in\mathbb{N}$, endowed with the norms
   \begin{equation*}
   \|f\|_{W_p^n(\mathbb{T})}:=\left( \sum_{k=0}^{n} \|f^{(k)}\|_{L^p(\mathbb{T})}^p \right)^{1/p},
   \end{equation*}
   are homogeneous and separable.

\item[8.] The \emph{fractional Sobolev spaces} $H_p^s(\mathbb{T})$, with $0<s\notin\mathbb{N}$ and $1<p<\infty$. They are obtained by the complex interpolation between integer order Sobolev spaces:
       \begin{equation*}
       H_p^s(\mathbb{T})=[W_p^{[s]}(\mathbb{T}), W_p^{[s]+1}(\mathbb{T})]_{\{s\}}
       \end{equation*}
       up to equivalence of norms, which follows from
       \cite[Theorem 2.4.2/1(d)]{Triebel}. (As usual, $[s]$ is the integral part of $s$, $\{s\}=s-[s]$, and $W_p^{0}(\mathbb{T}):=L^{p}(\mathbb{T}))$.

\item[9.] The \emph{Besov spaces} $B_{p,q}^{s}(\mathbb{T})$, with  $0<s<\infty$, $1<p<\infty$, and $1\leq q<\infty$. They are obtained by the real interpolation between Sobolev spaces; namely if $0<s_0<s_1$, $0<\theta<1$, and $s=(1-\theta)s_0+\theta s_1$, then
       \begin{equation*}
       B_{p,q}^{s}(\mathbb{T})=(H_{p}^{s_0}(\mathbb{T}), H_{p}^{s_1}(\mathbb{T}))_{\theta,q}
       \end{equation*}
       up to equivalence of norms, which follows from
       \cite[Theorem 2.4.2/2]{Triebel}.
\end{itemize}

\begin{remark}
All homogeneous Banach spaces from these examples are separable. This is not accidental. Indeed, if $X$ is a homogeneous Banach space, then the set of all trigonometric polynomials belonging to $X$ is dense in $X$, which     follows from the abstract version of Fejer's theorem considered in the next section. This implies the separability of~$X$.
\end{remark}

\section{Fourier series in homogeneous Banach spaces}\label{sect4}

Given a function $f\in L^1(\mathbb{T})$, we use the standard notation for its Fourier coefficients
\begin{equation*}
\widehat{f}(k):=\frac{1}{2\pi}\int\limits_{-\pi}^{\pi}f(s)e^{-iks}ds,
\;\;\mbox{with}\;\;k\in\mathbb{Z},
\end{equation*}
Fourier series
\begin{equation}\label{Fourier-series}
f(x)\sim\sum_{k\in \mathbb{Z}}\widehat{f}(k)e^{ikx},
\end{equation}
and its symmetric partial sums
\begin{equation*}
S_n(x,f):=\sum_{|k|\leq n}\widehat{f}(k)e^{ikx}
\;\;\mbox{of}\;\;x\in\mathbb{T},
\;\;\mbox{with}\;\;0\leq n\in\mathbb{Z}.
\end{equation*}
Hence,
\begin{equation*}
\sigma_n(x,f):=\frac{S_0(x,f)+S_1(x,f)+ ... +S_n(x,f)}{n+1}.
\end{equation*}

If a function $f$ belongs to a homogeneous Banach space $X$, then $\sigma_n(\cdot,f)\in X$ whenever $0\leq n\in\mathbb{Z}$ due to \cite[Chapter~I, Theorem~2.11]{Katznelson}; hence, every $S_n(x,f)\in X$ as well. Therefore, it makes sense to study conditions on the function $f$ for its Fourier series to converge to $f$ in~$X$. Here, the classical case where $X=C(\mathbb{T})$ is a natural guide.

Recall (see, e.g., \cite[Chapter~II, Section~2.2]{Katznelson}) that Hardy's Tauberian theorem for Fourier series states the following: if a function $f\in L^1(\mathbb{T})$ satisfies the condition
\begin{equation}\label{Hardy-cond-Fourier}
|\widehat{f}(k)|= O\left(\frac{1}{|k|}\right)\quad\mbox{as}\quad|k|\to\infty,
\end{equation}
then $\sigma_n(x,f)$ and $S_n(x,f)$ converge for the same values of $x$ to the same limit; also, if in addition $\sigma_n(x,f)$ converges uniformly on a certain subset of $\mathbb{T}$, so does $S_n(x,f)$ on this subset.
According to the Fej\'{e}r theorem, $\sigma_n(x,f)$ converges uniformly to $f$ on $\mathbb{T}$ whatever $f\in C(\mathbb{T})$. Hence, if a function $f\in C(\mathbb{T})$ satisfies \eqref{Hardy-cond-Fourier}, then the Fourier series \eqref{Fourier-series} converges to $f$ uniformly on $\mathbb{T}$ (i.e. in the Banach space $C(\mathbb{T})$). This result, called Hardy's test, is a generalization of the Dirichlet--Jordan test for the uniform convergence of Fourier series because every function $f$ of bounded variation on $\mathbb{T}$ satisfies \eqref{Hardy-cond-Fourier} (see, e.g., \cite[Capter~II, Theorems 4.12 and 8.1]{Zygmund}).

Of course, convergence tests expressed in terms of the Fourier coefficients  are of special interest. Such tests for the uniform convergence can be proved with the help of Tauberian theorems for number series. To obtain a test for convergence in homogeneous Banach spaces, we combine an abstract version of the Fej\'{e}r theorem for these spaces with Tauberian Theorem~\ref{thm:theor1}.

The abstract version of the Fej\'{e}r theorem states the following: if $X$ is a homogeneous Banach space endowed with a norm $\|\cdot\|$, then
\begin{equation*}
\lim_{n\to\infty}\|f(\cdot)-\sigma_n(\cdot,f)\|=0
\quad\mbox{for every}\quad f\in X.
\end{equation*}
This results from \cite[Chapter~I, Theorem~2.11]{Katznelson} and the fact that the Fej\'{e}r kernels form an approximate identity (in other words, form a summability kernel) (see, e.g., \cite[Section 5.1.1]{Edwards}).

\begin{theorem}\label{thm:theor2}
Let a Banach space $X$ be a linear subspace of $L^1(\mathbb{T})$ such that $f(\cdot+a)\in X$ whenever $f\in X$ and $a\in\mathbb{T}$ and that the function $e^{ikx}$ of $x$ belongs to $X$ whenever $k\in\mathbb{Z}$. Suppose that the norm in $X$ satisfies conditions $(H_1^*)$, $(H_3^*)$, and
\begin{equation}\label{cond-exp}
\|e^{ikx}\|_{X}=O(|k|^\alpha)\quad\mbox{as}\quad|k|\to\infty
\end{equation}
for a certain number $\alpha\geq0$. Then, if a function $f\in X$
satisfies the condition
\begin{equation}\label{cond-Fourier-coeff}
\sum_{|k|\geq n+1}k^{\alpha p}\,|\widehat{f}(k)|^p=
O\left(\frac{1}{n^{p-1}}\right)\quad\mbox{as}\quad n\to\infty
\end{equation}
for a certain number $p\geq1$, then the Fourier series \eqref{Fourier-series} converges to $f$ in the Banach space~$X$.
	\end{theorem}

\begin{proof}
Given $f\in X$, we consider the functions $u_{0}(x)\equiv\widehat{f}(0)$ and $u_{k}(x):=\widehat{f}(k)e^{ikx}+\widehat{f}(-k)e^{-ikx}$ of $x\in\mathbb{T}$ for each $k\in\mathbb{N}$. Owing to Theorem~\ref{thm:theor2_dodana}, $X$ is a homogeneous Banach space with respect to a certain equivalent norm $\|\cdot\|$. Hence, it follows from hypotheses \eqref{cond-exp} and \eqref{cond-Fourier-coeff} that the series $\sum_{k=0}^\infty u_{k}$ satisfies condition \eqref{eq: eq1}. This series is $(C,1)$-summable in $X$ to $f$ due to the abstract Fej\'{e}r theorem. Now the conclusion follows from Theorem~\ref{thm:theor1}.
\end{proof}

\begin{remark}
If the norm in a Banach space $X\subset L^1(\mathbb{T})$ satisfies $\|f\|_{X}=\|\,|f|\,\|_{X}$ whenever $f\in X$, then condition \eqref{cond-exp} is fulfilled for $\alpha=0$ (provided that the functions $e^{ikx}$ belong to $X$). Specifically, this is true for $C(\mathbb{T})$, the Lebesgue spaces and the Orlicz spaces. If $X$ is a symmetric Banach space over $\mathbb{T}$, then the Lebesgue space $L^{\infty}(\mathbb{T})$ is continuously embedded in $X$ (cf. \cite[Chapter~II, Theorem~4.1]{KreinPetuninSemenov}), which implies \eqref{cond-exp} for $\alpha=0$. Besides, condition \eqref{cond-exp} is fulfilled with $\alpha=n$ for $C^{n}(\mathbb{T})$ and $W_p^n(\mathbb{T})$ with the integer $n\geq1$.
\end{remark}

If $\alpha=0$, then Hardy's condition \eqref{Hardy-cond-Fourier} implies \eqref{cond-Fourier-coeff}. Thus, Theorem \ref{thm:theor2} can be viewed as a broad generalization of Hardy's test.

\section{A refined form of Cantor's theorem}\label{sect5}

Let us prove a result which will be used in the next section. The result may be also of independent interest.

Let $(X,d)$ and $(Y,\rho)$ be metric spaces. Let us introduce the following notion:

\begin{definition}
A mapping $f:X \rightarrow Y$ is called uniformly continuous on a set
$A\subset X$ over the space $X$ if for every number $\varepsilon>0$ there exists a number $\delta>0$ such that
\begin{equation*}
(x\in A, x'\in X, d(x, x')<\delta)\Longrightarrow
\rho(f(x),f(x'))<\varepsilon.
\end{equation*}
\end{definition}

If $A$ consists of a single point, the above definition means the continuity of $f$ at this point. If $A=X$, this definition coincides with the known one.

\begin{theorem}\label{thm:theor3}
If a mapping $f:X \rightarrow Y$ is continuous on a compact set $A\subset X$, then $f$ is uniformly continuous on $A$ over the space $X$. 	
\end{theorem}

The proof is quite standard. We give it for the reader's convenient.

\begin{proof}
Choose a number $\varepsilon>0$ arbitrarily. Since $f$ is continuous at every point $y\in A$, there exists a number $\delta(y)>0$ such that
\begin{equation}\label{cond-continuity}
\bigl(y'\in X, d(y,y')<\delta(y)\bigr)\Longrightarrow\rho(f(y),f(y'))< \frac{\varepsilon}{2}.
\end{equation}
The set $\{B(y,\delta(y)/2):y\in A\}$ of open balls form an open cover of the compact set $A$. Hence, the cover contains a finite subcover consisting of the balls centered at some points $y_{1}$,..., $y_{n}\in A$.
	
Put
\begin{equation*}
\delta:=\frac{1}{2}\min\{\delta(y_{k}):1\leq k\leq n\}.
\end{equation*}
Suppose that points $x\in A$ and $x'\in X$ satisfy $d(x,x')< \delta$. Since $x\in B(y_{k},\delta(y_{k})/2)$ for certain $k\in\{1,\ldots,n\}$, we have $x'\in B(y_k,\delta(y_k))$. Thus,
\begin{equation*}
\rho(f(x),f(x'))\leq\rho(f(y_k),f(x))+\rho(f(y_k),f(x'))\leq \frac{\varepsilon}{2}+\frac{\varepsilon}{2}=\varepsilon
\end{equation*}
due to \eqref{cond-continuity} with $y:=y_k$.
\end{proof}

\section{The uniform convergence and summability on a set}\label{sect6}

Let $f\in L^1(\mathbb{T})$. The Lebesgue theorem states that $\sigma_{n}(x,f)\to f(x)$ as $n\to\infty$ if $x$ is a Lebesgue point of $f$ (see, e.g., \cite[Chapter~III, Theorem~3.9]{Zygmund}). Specifically, this convergence holds true at every point $x\in C_f$. As usual, $C_f$ stands for the set of all points where $f$ is continuous. Hence, if the Fourier coefficients of $f$ satisfy \eqref{Hardy-cond-Fourier}, then $S_n(x, f)\to f(x)$ as $n\to\infty$ for every point $x\in C_f$ by Hardy's Tauberian theorem for Fourier series.

It is natural to raise the question of describing the sets where the last convergence is uniform. This question is significant because there are continuous functions on $\mathbb{T}$ whose Fourier series converge at every point of $\mathbb{T}$ but the convergence failed to be uniform on any closed subinterval of~$\mathbb{T}$. An example of such a function is given in \cite[Chapter~I, Section~44]{Bari}.

It follows from \cite[Chapter~III, Theorem~3.4]{Zygmund}, that \eqref{Hardy-cond-Fourier} implies the uniform convergence of $S_n(x, f)$ to $f(x)$ on every connected closed subset of $C_f$.  We will improve this result. Since the uniform convergence of a sequence of functions $f_n\in C(\mathbb{T})$ on a set $K\subset\mathbb{T}$ implies this convergence on the closure of $K$, we  restrict ourselves to closed subsets of $\mathbb{T}$. For our purpose, we need a certain extension of the classical Fej\'{e}r theorem.

Recall \cite[Chapter~III, Theorem~3.4]{Zygmund} that the Fej\'{e}r kernel
\begin{equation*}
K_{n}(t)= \frac{1}{n+1}\sum_{k=0}^{n}\frac{\sin(k+1/2)t}{2\sin(t/2)}
\end{equation*}
corresponds to the $(C,1)$ summation method of trigonometric Fourier series. This kernel possesses the following properties:
\begin{itemize}
    \item[(A)] $\frac{1}{\pi} \int\limits_{-\pi}^\pi K_n(t)dt=1$ \,whenever $n\in\mathbb{N}$;
    \item[(B)] $K_n(t)\geq0$ whenever $n\in\mathbb{N}$;
    \item[(C)] $\mu_n(\delta)\to0$ as $n\to\infty$ for every   $\delta\in(0,\pi]$, with $\mu_n(\delta):=\max\{K_n(t):\delta\leq t\leq \pi\}$.
\end{itemize}

\begin{theorem}\label{thm:theor4}
Let $f\in L^1(\mathbb{T})$, and let $K$ be a closed subset of $C_f$. Then
$\sigma_n(x,f)$ converges uniformly to $f(x)$ on $K$ as $n\to\infty$.
\end{theorem}

If $C_f=K=\mathbb{T}$, we arrive at the classical Fej\'{e}r theorem.

\begin{proof}
Given $x,t\in\mathbb{T}$, we put
\begin{equation*}
\varphi_x(t):=\frac{1}{2}\bigl(f(x+t)+f(x-t)-2f(x)\bigr).
\end{equation*}
Since the function $f$ is continuous on the compact set $K\subset\mathbb{T}$, Theorem~\ref{thm:theor3} asserts that $f$ is uniformly continuous on $K$ over the metric space $\mathbb{T}$. Hence,
for every number $\varepsilon>0$, there exists a number $\delta=\delta(\varepsilon)>0$ such that
\begin{equation*}
|\varphi_{x_0}(t)|\leq \frac{1}{2}\left| f(x_0+t)-f(x_0)\right| + \frac{1}{2}\left| f(x_0-t)-f(x_0)\right| <\varepsilon
\end{equation*}
whenever $|t|<\delta$ and $x_0\in K$. Now we may repeat the reasoning given in the proof of \cite[Chapter~III, Theorem~3.4]{Zygmund} to substantiate Theorem~\ref{thm:theor4}.
 \end{proof}

\begin{remark}
Theorem \ref{thm:theor4} allows generalizing to an arbitrary summation method whose kernel satisfies conditions (A), (B), and (C). Moreover \cite[Chapter~III, Theorem~2.21]{Zygmund}, condition (B) can be replaced with the following more general (in view of (A)) condition:
\begin{itemize}
    \item[$(\mathrm{B}')$] there exists a number $C>0$ such that $\int\limits_{-\pi}^\pi|K_n(t)|dt\leq C$ whenever $n\in \mathbb{N}$
\end{itemize}
\end{remark}

\begin{theorem}\label{thm:theor5}
Suppose that a function $f\in L^1(\mathbb{T})$ satisfies the condition
\begin{equation}\label{cond-Fourier-coeff-bis}
\sum_{|k|\geq n+1}|\widehat{f}(k)|^p=
O\left(\frac{1}{n^{p-1}}\right)\quad\mbox{as}\quad n\to\infty
\end{equation}
for a certain number $p>1$. Then
$S_n(x,f)$ converges uniformly to $f(x)$, as $n\to\infty$, on every closed set $K\subset C_f$.
\end{theorem}

\begin{proof} This follows from Theorem~\ref{thm:theor4} and Theorem~\ref{thm:theor1} considered in the $X=C(K)$ case.
\end{proof}

\begin{remark}
Theorem \ref{thm:theor1} implies the following generalization of Hardy's Tauberian theorem for Fourier series: if a function $f\in L^1(\mathbb{T})$ satisfies \eqref{cond-Fourier-coeff-bis} for a certain number $p>1$ and if $\sigma_n(x,f)$ converges uniformly on a certain subset $K$ of $\mathbb{T}$, then $S_n(x,f)$ converges uniformly on the closure $\overline{K}$ of $K$. We note that $\sigma_n(x,f)$ converges uniformly on $\overline{K}$ and uses Theorem~\ref{thm:theor1} in the case when $X=C(\overline{K})$ and where $u_{0}(x)=\widehat{f}(0)$ and $u_{k}(x)=\widehat{f}(k)e^{ikx}+\widehat{f}(-k)e^{-ikx}$ for all $x\in\overline{K}$ and $k\in\mathbb{N}$.
\end{remark}

\section*{Statements and declarations}

\subsection*{Funding.} This work was funded by the Czech Academy of Sciences within grant RVO:67985840 and by the National Academy of Sciences of Ukraine within project 0120U100169. The authors was also supported by the European Union's Horizon 2020 research and innovation programme under the Marie Sk{\l}odowska-Curie grant agreement No~873071 (SOMPATY: Spectral Optimization: From Mathematics to Physics and Advanced Technology).

\subsection*{Author contribution.} Conceptualization by Vladimir Mikhailets; Sections~\ref{sect1} and~\ref{sect4} are performed by Vladimir Mikhailets and Aleksandr Murach; Section~\ref{sect2} by Vladimir Mikhailets and Oksana Tsyhanok; Section~\ref{sect3} by Aleksandr Murach;   Sections~\ref{sect5} and~\ref{sect6} by Vladimir Mikhailets.

\end{document}